\documentclass[11pt]{article}

\usepackage{amsxtra,amssymb,amsthm,amsmath,latexsym}
\usepackage{graphicx}
\usepackage{epsfig}
\usepackage{afterpage}
\newtheorem{thm}{Theorem}[]

\newtheorem{rem}[]{Remark}
 \newcommand{\thmref}[1]{Theorem~\ref{#1}}

\newcommand{\R}{{\mathbb R}}

\newcommand{\bee}{\begin{equation*}}
\newcommand{\eee}{\end{equation*}}
\newcommand{\be}{\begin{equation}}
\newcommand{\ee}{\end{equation}}
\newcommand{\pn}{\par\noindent}

\title{Stability of solutions to some evolution problems}
\author{A.G. Ramm \\
\small Department of Mathematics\\[-0.8ex]
\small Kansas State University, Manhattan, KS 66506-2602, USA\\
\small \texttt{ramm@math.ksu.edu}}

\date{}
\begin{document}
Plenary talk at the Chaos 2011 International Conference.

In the book "Topics on Chaotic Systems: Selected Papers from Chaos 2011
International Conference" Editors C.Skiadas, I.Dimotikalis,
Char.Skiadas, World Sci. Publishing, 2012

A. G. Ramm, Stability of solutions to some evolution problems.
\maketitle

\begin{abstract}
Large time behavior of solutions to abstract differential equations 
is studied. The corresponding evolution problem is:
$$\dot{u}=A(t)u+F(t,u)+b(t), \quad t\ge 0; \quad u(0)=u_0. \qquad (*)$$
Here $\dot{u}:=\frac {du}{dt}$, $u=u(t)\in H$, $t\in \R_+:=[0,\infty)$,
$A(t)$ is a linear dissipative operator: Re$(A(t)u,u)\le -\gamma(t)(u,u)$, 
$\gamma(t)\ge 0$, $F(t,u)$ is a nonlinear operator,
$\|F(t,u)\|\le c_0\|u\|^p$, $p>1$, $c_0,p$ are constants, $\|b(t)\|\le 
\beta(t),$ $\beta(t)\ge 0$  
is a continuous function.  

Sufficient conditions are given for the solution $u(t)$
to problem (*) to exist for all $t\ge0$, to be bounded uniformly on
$\R_+$, and a bound on $\|u(t)\|$ is given. This bound implies 
the relation $\lim_{t\to \infty}\|u(t)\|=0$ under suitable
conditions on $\gamma(t)$ and $\beta(t)$.

The basic technical tool in this work is the following
nonlinear inequality:
$$ \dot{g}(t)\leq
-\gamma(t)g(t)+\alpha(t,g(t))+\beta(t),\ t\geq 0;\quad g(0)=g_0,
$$
which holds on any interval $[0,T)$ on which $g(t)\ge 0$ exists and has 
bounded derivative
from the right, $\dot{g}(t):=\lim_{s\to +0}\frac{g(t+s)-g(t)}{s}$.  
It is assumed that  $\gamma(t)$, 
and $\beta(t)$ are nonnegative continuous 
functions of $t$ defined on
$\R_+:=[0,\infty)$,  the function $\alpha(t,g)$ is
defined for all $t\in \R_+$,  locally Lipschitz 
with
respect to $g$ uniformly with respect to $t$ on any compact subsets
$[0,T]$, $T<\infty$,   and non-decreasing with respect to $g$. 
If 
there
exists a function $\mu(t)>0$,  $\mu(t)\in C^1(\R_+)$, such that
$$\alpha\left(t,\frac{1}{\mu(t)}\right)+\beta(t)\leq
\frac{1}{\mu(t)}\left(\gamma(t)-\frac{\dot{\mu}(t)}{\mu(t)}\right),\quad
\forall t\ge 0;\quad \mu(0)g(0)\leq 1,$$ then $g(t)$ exists on all of 
$\R_+$, that is $T=\infty$,  and the following estimate holds:
$$0\leq g(t)\le \frac 1{\mu(t)},\quad \forall t\geq 0. $$ 
If $\mu(0)g(0)< 1$, then $0\leq g(t)< \frac 1{\mu(t)},\quad \forall t\geq 
0. $

\end{abstract}
\pn{

\pn {\em PACS: }\,\, 02.30Tb. \\

 }\\

\section{Introduction}
Consider an abstract nonlinear
evolution problem
\be\label{e1}
\dot{u}=A(t)u+F(t,u)+b(t),  \qquad \dot{u}:=\frac {du}{dt},
\ee
\be\label{e2}
u(0)=u_0,\ee
where $u(t)$ is a function with values in a Hilbert space $H$, $A(t)$
is a linear bounded dissipative operator in $H$, which satisfies 
inequality
\be\label{e3} \text{Re}(A(t)u,u)\leq -\gamma(t)\|u\|^2,\qquad t\geq 
0;\qquad \forall u\in H,
\ee
where 
$F(t,u)$ is a nonlinear map in $H$, 
\be\label{e4} \|F(t,u)\|\leq c_0\|u(t)\|^p,\qquad p>1, 
\ee
\be\label{e5}
\|b(t)\|\leq \beta(t),
\ee
$\gamma(t)>0$ and $\beta(t)\ge 0$ are continuous function, defined on all 
of $\R_+:=[0,\infty)$, $c_0>0$ and $p>1$ are constants. 

Recall that a linear operator $A$ in a Hilbert space is called dissipative
if Re$(Au,u)\le 0$ for all $u\in D(A)$, where $D(A)$ is the domain of 
definition of $A$. Dissipative operators are important because they
describe systems in which energy is dissipating, for example, due to 
friction or other physical reasons. Passive nonlinear networks
can be described by equation \eqref{e1} with a dissipative 
linear operator $A(t)$, see \cite{R129} and \cite{R118}, Chapter 3.

Let $\sigma:=\sigma(A(t))$ denote the spectrum of the linear operator
$A(t)$,  $\Pi:=\{z: Re z<0\}$, $\ell:=\{z: Re z=0\}$, and $\rho(\sigma, 
\ell)$  denote the distance between sets $\sigma$ and $\ell$. 
We assume that 
\be\label{e6}
\sigma\subset \Pi,
\ee 
but we allow $\lim_{t\to \infty}
\rho(\sigma,\ell)=0$. This is the basic novel point in our theory.
The usual assumption in stability theory (see, e.g., \cite{DK}) is
$ \sup_{z\in \sigma} Re z\le -\gamma_0$, where $\gamma_0=const>0$.
For example, if $A(t)=A^*(t)$, where $A^*$ is the adjoint operator,
and if the spectrum of $A(t)$ consists of eigenvalues $\lambda_j(t)$, 
$0\ge \lambda_j(t)\ge \lambda_{j+1}(t)$, then, 
we allow $\lim_{t\to \infty} \lambda_1(t)=0$.
This is in contrast with the usual
theory, where the assumption is $\lambda_1(t)\le -\gamma_0$,
$\gamma_0>0$ is a constant, is used. 

Our goal is to give 
sufficient conditions for the existence and uniqueness of the solution
to problem \eqref{e1}-\eqref{e2} for all $t\ge0$, that is, for global 
existence of $u(t)$, for boundedness of $\sup_{t\ge0}\|u(t)\|<\infty,$
or to the relation $\lim_{t\to \infty}\|u(t)\|=0$.

If $b(t)=0$ in \eqref{e1}, then $u(t)=0$ solves equation \eqref{e1}
and $u(0)=0$. This equation is called zero solution to \eqref{e1}
with $b(t)=0$.  

Recall that the zero solution to
equation \eqref{e1} with $b(t)=0$ is called Lyapunov stable if for any 
$\epsilon>0$, however small, one can find a 
$\delta=\delta(\epsilon)>0$, 
such that if
$\|u_0\|\le \delta$, then the solution to Cauchy problem
\eqref{e1}-\eqref{e2} satisfies the estimate $\sup_{t\ge 0}\|u(t)\|\le
\epsilon$. If, in addition, $\lim_{t\to \infty}\|u(t)\|=0$, then
the zero solution to equation \eqref{e6} is called asymptotically stable
in the Lyapunov sense.

If $b(t)\not\equiv 0$, then one says that  \eqref{e1}-\eqref{e2} is the 
problem with persistently acting perturbations.     
The zero solution is called Lyapunov stable for problem 
\eqref{e1}-\eqref{e2}
with persistently acting perturbations if for any $\epsilon>0$, however 
small, one can find a $\delta=\delta(\epsilon)>0$, such that if
$\|u_0\|\le \delta$, and $\sup_{t\ge 0}\|b(t)\|\le\delta$, then the 
solution to Cauchy problem
\eqref{e1}-\eqref{e2} satisfies the estimate $\sup_{t\ge 0}\|u(t)\|\le
\epsilon$.  

The approach, developed in this work, consists of reducing the 
stability problems to some nonlinear differential inequality
and estimating the solutions to this inequality.

In Section 2 the formulation and a proof of two theorems, containing the 
result concerning this inequality and its discrete analog, are given.
In Section 3 some results concerning Lyapunov stability  
of zero  solution to equation  \eqref{e1} are obtained.

The results of this paper are based on the works \cite{R593}-
\cite{R118}.

In the theory of chaos one of the reasons for the chaotic behavior
of a solution to an evolution problem to
appear is the lack of stability of solutions to this problem
(\cite{Da}, \cite{De}). The results presented in Section 3 can be 
considered as sufficient conditions for chaotic behavior not to appear
in the evolution system described by problem \eqref{e1}-\eqref{e2}.

\section{Differential inequality}

In this Section a self-contained proof is given of an
estimate for solutions of a nonlinear inequality 
\be\label{e7}
\dot{g}(t)\leq -\gamma(t)g(t)+\alpha(t,g(t))+\beta(t),\ t\geq 0;\
g(0)=g_0;\quad \dot{g}:=\frac{dg}{dt}.  \ee
In Section 3 some of many possible applications of this estimate 
(estimate  \eqref{e11})  are demonstrated.

It is  not assumed a priori that solutions $g(t)$ to inequality
\eqref{e7} are defined on all of $\R_+$, that is, that  these solutions 
exist globally. In Theorem 1 we give sufficient conditions for the 
global existence of 
$g(t)$.
Moreover, under these conditions  a bound on $g(t)$ is given, see
estimate \eqref{e11} in Theorem 1. 
This bound yields the relation $\lim_{t\to \infty}g(t)=0$
if $\lim_{t\to \infty}\mu(t)=\infty$ in \eqref{e11}. 

Let us formulate
our assumptions.

\noindent ${\it Assumption\,\, A_1).}$ We assume that the function 
$g(t)\geq 
0$
is defined on some interval $[0,T)$, has a bounded derivative
$\dot{g}(t):=\lim_{s\to +0}\frac{g(t+s)-g(t)}{s}$ from the 
right at any point of this interval, and  $g(t)$ satisfies 
inequality \eqref{e7} at all $t$ at which $g(t)$ is defined. The 
functions 
$\gamma(t)$, and $\beta(t)$,  are continuous, non-negative,
defined on all of $\R_+$.  The function 
$\alpha(t,g)\ge 0$ is 
continuous on $\R_+\times \R_+$, nondecreasing with respect to $g$,
and locally Lipschitz with respect to $g$. This 
means that 
$$\alpha(t,g)\ge \alpha(t,h)\quad \text{if}	\quad g\ge h,$$ and 
\be\label{e8}
|\alpha(t,g)-\alpha(t,h)|\leq L(T,M)|g-h|,
\ee
if $t\in[0,T]$, $|g|\leq M$ and $|h|\leq  M$, $M=const>0$,
where $L(T,M)>0$ is a constant independent of $g$, $h$, and $t$.

\noindent ${\it Assumption\,\, A_2).}$ There exists a $C^1(\R_+)$
function $\mu(t)>0$,  such that 
\be\label{e9}
\alpha\left(t,\frac{1}{\mu(t)}\right)+\beta(t)\leq
\frac{1}{\mu(t)}\left(\gamma(t)-\frac{\dot{\mu}(t)}{\mu(t)}\right),\quad 
\forall t\ge 0, \ee
and
\be\label{e10} \mu(0)g(0)< 1. \ee

If $\mu(0)g(0)\le 1$, then the 
inequality  $g(t)< \frac 1 {\mu(t)}$ in Theorem 
1 in formula \eqref{e11} should be replaced by $g(t)\le \frac 1 {\mu(t)}$. 

\begin{thm}\label{thm1}
If ${\it Assumptions\,  A_1)}$ and ${\it A_2)}$ hold, then any solution 
$g(t)\ge 0$ 
to inequality \eqref{e7} exists on all of $\R_+$, i.e., $T=\infty$, 
and satisfies the following estimate:
\be\label{e11}0\leq g(t)<\frac{1}{\mu(t)}\quad \forall t\in \R_+. \ee
If $ \mu(0)g(0)\le 1$, then $0\leq g(t)\le \frac{1}{\mu(t)}\quad \forall 
t\in \R_+.$
\end{thm}
\begin{rem}\label{rem1}
If $\lim_{t\to \infty} \mu(t)=\infty$, then $\lim_{t\to
\infty}g(t)=0$. 
\end{rem}

\noindent {\it Proof of \thmref{thm1}.} Let 
\be\label{e12}
g(t)=\frac{v(t)}{a(t)},\quad a(t):=e^{\int_0^t\gamma(s)ds}, \ee
\be\label{e13} \eta(t):=\frac{a(t)}{\mu(t)},\quad
\eta(0)=\frac{1}{\mu(0)}>g(0). \ee Then
inequality \eqref{e7} reduces to 
\be\label{e14} \dot{v}(t)\leq
a(t)\alpha\left(t,\frac{v(t)}{a(t)}\right)+a(t)\beta(t),\quad t\geq 0; 
\quad
v(0)=g(0). \ee 
One has 
\be\label{e15}
\dot{\eta}(t)=\frac{\gamma(t)a(t)}{\mu(t)}-\frac{\dot{\mu}(t)a(t)}{\mu^2(t)}
=\frac{a(t)}{\mu(t)}\left(\gamma(t)-\frac{\dot{\mu}(t)}{\mu(t)}\right).
\ee 
From \eqref{e9}, \eqref{e14}-\eqref{e15}, one gets
\be\label{e16} v(0)<\eta(0), \qquad
\dot{v}(0)\le \dot{\eta}(0), 
\ee
so there is an interval $[0,T)$ such that
\be\label{e17}
0\leq v(t)<\eta(t),\quad \forall t\in[0,T). \ee 
Inequality \eqref{e17} is equivalent to the inequality
\be\label{e18}
0\leq g(t)<\frac 1 {\mu(t)}, \qquad \forall t\in[0,T).\ee 
Inequality \eqref{e17} holds on the maximal interval $[0,T_{max})$ of the 
existence of $v$. Indeed, from \eqref{e18}, \eqref{e9}, \eqref{e14}, 
\eqref{e15}, and the assumption
that $\alpha(t,g)$ is nondecreasing with respect to $g$, it follows that
\eqref{e17} implies $\dot{v}(t)\leq \dot{\eta}(t)$ for $t\in [0,T)$.
Integrating this inequality, one gets $v(T)+\eta(0)-v(0)\leq \eta(T)$.
Since $\eta(0)=\frac 1 {\mu(0)}>g(0)=v(0)$, it follows that 
$v(T)<\eta(T)$. Thus, one can argue 
as before, replacing $0$ by $T$,  and increase the interval $[0,T)$ to the 
maximal interval $[0,T_{max})$ of the existence of $v$. We denote 
$T_{max}$ by $T$. 

Let us prove that $T=\infty$.  The right-hand side
of inequality \eqref{e17} is defined for all  $t\ge 0$.
The function $g(t)$, a solution to inequality \eqref{e7}, exists
on every interval on which $v(t)$ exists, and $v(t)$, the solution to
inequality \eqref{e14}, exists on every interval on which the solution
$w(t)$ to the problem
\begin{equation}
\label{e19}
\dot{w}(t)=a(t)\left[\alpha\left(t,\frac {w(t)}{a(t)}\right) + \beta 
(t)\right],\qquad
w(0)=v(0),
\end{equation}
exists.

We have proved that the solution to problem \eqref{e19}
(which  is a solution to problem \eqref{e14} as well) satisfies
the estimate
\begin{equation}
\label{e20}
0\le w(t)\le \frac {a(t)}{\mu(t)}
\end{equation}
on every interval $[0,T)$ on
which  $w$ exists. We claim that {\it estimate \eqref{e20} implies that 
$w$
exists for all $t\ge 0$, in other words, that $T=\infty$.}

Indeed, according to the known result (see, e.g., \cite{H},
Theorem 3.1 in Chapter 2), if 
the maximal interval $[0,T)$ of the
existence of the solution to problem \eqref{e19} is finite,
that is $T<\infty$, then $\lim_{t\to T-0}w(t)=\infty$. This, however, 
cannot happen because
in the inequality  \eqref{e20} the function $\frac 
{a(t)}{\mu(t)}$
is bounded for every $t\ge 0$.

Another argument, proving that $T=\infty$, can be given.
It follows from \eqref{e17} that there is a sequence $t_n<T$, 
$\lim_{n\to \infty}t_n= T$, such that $\lim_{n\to \infty}v(t_n)$
exists and is finite. Let us denote $\lim_{n\to \infty}v(t_n)=v(T)$,
$v(T)\leq \eta(T)$. This limit does not depend on the choice of the
sequence $t_n$, converging to $T$, because the derivative of 
$v(t)$ is bounded on $[0,T)$.
Since $\alpha(t,\frac {w(t)}{a(t)})$ is locally 
Lipschitz with respect to $w$, there exists a solution to 
equation \eqref{e19} with the initial data $w(T)=v(T)$ on the interval
$[T,T+\epsilon)$, for some $\epsilon>0$. Therefore $[0,T]$ is not the 
maximal interval of the existence of $w(t)$, unless $T=\infty$.

\thmref{thm1} is proved.\hfill $\Box$

Let us formulate and prove a discrete version of \thmref{thm1}.

\begin{thm}\label{thm2}
Assume that $g_n\geq 0$, $\alpha(n,g_n)\geq 0,$ 
\be\label{e21}
g_{n+1}\leq (1-h_n\gamma_n)g_n+h_n\alpha(n,g_n)+h_n\beta_n;\quad
h_n>0,\quad 0<h_n\gamma_n<1,\ee and $\alpha(n,g_n)\geq \alpha(n,p_n)$ if
$g_n\geq p_n$. If there exists a sequence $\mu_n> 0$ such that
\be\label{e22} \alpha(n,\frac{1}{\mu_n}) +\beta_n\leq
\frac{1}{\mu_n}(\gamma_n-\frac{\mu_{n+1}-\mu_n}{h_n\mu_n}),
\ee
and
\be\label{e23} g_0\leq \frac{1}{\mu_0}, \ee 
then 
\be\label{e24}
0\leq g_n\leq \frac{1}{\mu_n}, \qquad \forall n\geq 0. \ee
\end{thm}
\begin{proof}
For $n=0$ inequality \eqref{e24} holds because of \eqref{e23}.
Assume that it holds for all $n\leq m$ and let us check that then it holds
for $n=m+1$. If this is done, \thmref{thm2} is proved. 

Using the
inductive assumption, one gets: 
\bee g_{m+1}\leq
(1-h_m\gamma_m)\frac{1}{\mu_m}+h_m\alpha(m,\frac{1}{\mu_m})+h_m\beta_m.
\eee 
This and inequality \eqref{e22} imply: \bee\begin{split} g_{m+1}&\leq
(1-h_m\gamma_m)\frac{1}{\mu_m}+h_m\frac{1}{\mu_m}(\gamma_m-\frac{\mu_{m+1}-
\mu_m}{h_m\mu_m})\\
&=\frac{\mu_mh_m-\mu_mh^2_m\gamma_m+h^2_m\gamma_m\mu_m-h_m\mu_{m+1}+h_m\mu_m
}{\mu^2_mh_m}\\
&=\frac{2\mu_mh_m-h_m\mu_{m+1}}{\mu_m^2h_m}=\frac{2\mu_m-\mu_{m+1}}{\mu^2_m}=
\frac{1}{\mu_{m+1}}+\frac{2\mu_m-\mu_{m+1}}{\mu^2_m}-\frac{1}{\mu_{m+1}}.
\end{split}\eee 
The proof is completed if one checks that
\bee \frac{2\mu_m-\mu_{m+1}}{\mu^2_m}\leq \frac{1}{\mu_{m+1}}, \eee
or, equivalently, that 
\bee 2\mu_m\mu_{m+1}-\mu^2_{m+1}-\mu^2_m\leq 0. 
\eee 
The last
inequality is obvious since it can be written as 
$$-(\mu_m-\mu_{m+1})^2\le 0.$$
\thmref{thm2} is proved.
\end{proof}
\thmref{thm2} was formulated in \cite{R593} and proved in \cite{R558}.
We included for completeness a proof, which  differs from the one in 
\cite{R558} only slightly.

\section{Stability results}

In this Section we develop a method for a study of stability 
of solutions to the evolution problems described by the Cauchy
problem \eqref{e1}-\eqref{e2} for abstract differential equations
with a dissipative bounded  linear operator $A(t)$ and a nonlinearity 
$F(t,u)$ satisfying inequality \eqref{e4}. Condition \eqref{e4} means that 
for
sufficiently small $\|u(t)\|$ the nonlinearity is of the higher
order of smallness than $\|u(t)\|$. We also study the large time behavior 
of the solution to problem \eqref{e1}-\eqref{e2} with 
persistently acting perturbations $b(t)$. 

In this paper we assume that
$A(t)$ is a bounded linear dissipative operator, but our methods are
valid also for unbounded linear dissipative operators $A(t)$,
for which one can prove global existence of the solution to
problem \eqref{e1}-\eqref{e2}. We do not go into further detail
in this paper.

Let us formulate the first stability result. 

{\bf Theorem 3.} {\it Assume that Re$(Au,u)\le -k\|u\|^2$
$\forall u\in H$, $k=const>0$, and inequality \eqref{e3} holds with 
$\gamma(t)=k$. 
Then the solution
to problem \eqref{e1}-\eqref{e2} with $b(t)=0$ satisfies an esimate
$\|u(t)\|=O(e^{-(k-\epsilon)t)}$ as $t\to \infty$. Here
$0<\epsilon< k$ can be chosen arbitrarily small if
$\|u_0\|$ is sufficiently small.}

This theorem implies asymptotic stability in the sense of Lyapunov of
the zero solution to equation \eqref{e1} with $b(t)=0$. Our proof of 
Theorem 3 is new and very short.

{\it Proof of Theorem 3}.

Multiply equation \eqref{e1} (in which $b(t)=0$ is assumed) by $u$, denote
$g=g(t):=\|u(t)\|$, take the real part, and use  assumption
\eqref{e3} with $\gamma(t)=k>0$, to get
\begin{equation}
\label{e25}
g\dot{g}\le -kg^2+c_0g^{p+1}, \qquad p>1.
\end{equation}
If $g(t)>0$ then the derivative $\dot{g}$ does exist, and 
$$\dot{g}(t)=Re  \big(\dot{u}(t), \frac {u(t)}{\|u(t)\|}\big),$$ 
as one can check. 
If $g(t)=0$ on an open subset of $\R_+$, then
the derivative $\dot{g}$ does exist on this subset and $\dot{g}(t)=0$
on this subset. If $g(t)=0$ but in in any neighborhood $(t-\delta,
t+\delta)$ there are points at which $g$ does not vanish,
then by $\dot{g}$ we understand the derivative from the right,
that is,
$$\dot{g}(t):= \lim_{s\to +0}\frac {g(t+s)-g(t)}{s}=\lim_{s\to +0}\frac
{g(t+s)}{s}.$$
This limit does exist and is equal to $\|\dot{u}(t)\|$.
Indeed, the function $u(t)$ is continuously differentiable,
so
$$\lim_{s\to +0}\frac {\|u(t+s)\|}{s}=\lim_{s\to +0}
\frac{\|s\dot{u}(t)+o(s)\|}{s}=\|\dot{u}(t)\|.$$
The assumption about the existence of the bounded derivative $\dot{g}(t)$
from the right in Theorem 3 was made because the function $\|u(t)\|$
does not have, in general, the derivative in the usual sense at the points
$t$ at which $\|u(t)\|=0$, no matter how smooth the function $u(t)$
is at the point $\tau$. Indeed, 
$$\lim_{s\to -0}\frac {\|u(t+s)\|}{s}=\lim_{s\to -0}
\frac{\|s\dot{u}(t)+o(s)\|}{s}=-\|\dot{u}(t)\|,$$
because $\lim_{s\to -0}\frac {|s|}{s}=-1$. Consequently,
the right and left derivatives of $\|u(t)\|$ at the point $t$
at which $\|u(t)\|=0$ do exist, but are different. Therefore,
the derivative of $\|u(t)\|$ at the point $t$
at which $\|u(t)\|=0$ does not exist in the usual sense.

However, as we have proved above, the derivative
$\dot{g}(t)$ from the right does exist always, provided that $u(t)$ is 
continuously
differentiable at the point $t$.

Since $g\ge 0$, inequality \eqref{e25}  yields inequality \eqref{e7}
with
$\gamma(t)=k>0$, $\beta(t)=0$, and $\alpha(t,g)=c_0 g^p$, $p>1$.
Inequality \eqref{e9} takes the form
\begin{equation}
\label{e26}
\frac {c_0}{\mu^p(t)}\leq \frac 1 {\mu(t)}\left(k-\frac
{\dot{\mu}(t)}{\mu(t)} \right), \qquad  \forall t\ge 0.
\end{equation}
Let
\begin{equation}
\label{e27}
\mu(t)=\lambda e^{bt}, \qquad \lambda,b=const>0.
\end{equation}
We choose the constants $\lambda$ and $b$ later.
Inequality \eqref{e9}, with $\mu$ defined in \eqref{e27}, takes the form
\begin{equation}
\label{e28}
\frac {c_0}{\lambda^{p-1}e^{(p-1)bt}} +b\leq k, \qquad \forall
t\ge 0.
\end{equation}
This inequality holds if it holds at $t=0$, that is, if
\begin{equation}
\label{e29}
\frac {c_0}{\lambda^{p-1}} +b\leq k.
\end{equation}
Let $\epsilon>0$ be arbitrary small number. Choose
$b=k-\epsilon>0$. Then  \eqref{e29} holds if
\begin{equation}
\label{e30}
\lambda\geq \big(\frac{c_0}{\epsilon} \big)^{\frac 1 {p-1}}.
\end{equation}
Condition \eqref{e10} holds if
\begin{equation}
\label{e31}
\|u_0\|=g(0)\le \frac 1 {\lambda}.
\end{equation}
We choose $\lambda$ and $b$ so that inequalities \eqref{e30}
and \eqref{e31} hold. This is always possible if $b<k$ and $\|u_0\|$
is sufficiently small.

By Theorem 1, if  inequalities  \eqref{e29} and \eqref{e31} hold, then 
one gets estimate  \eqref{e11}:
\begin{equation}
\label{e32}
0\le g(t)=\|u(t)\|\le \frac {e^{-(k-\epsilon)t}} {\lambda},\qquad
\forall t\ge 0.
\end{equation}
Theorem 3 is proved. \hfill $\Box$

{\bf Remark 1.} {\it One can formulate the result differently.
Namely, choose  $\lambda=\|u_0\|^{-1}$. Then inequality \eqref{e31} 
holds, and becomes an equality.
Substitute this $\lambda$ into \eqref{e29} and get
$$c_0\|u_0\|^{p-1}+b\leq k.$$ 
Since the choice of the constant $b>0$ is
at our disposal, this inequality can always be satisfied if 
$c_0\|u_0\|^{p-1}<k$.
Therefore, condition 
$$c_0\|u_0\|^{p-1}<k$$ 
is a sufficient condition for
the estimate 
$$\|u(t)\|\le \|u_0\|e^{-(k-c_0\|u_0\|^{p-1})t},$$
to hold, provided  that $c_0\|u_0\|^{p-1}<k$. 
}

Let us formulate the second stability result. 

{\bf Theorem 4.} {\it Assume that inequalities \eqref{e3}-\eqref{e5} 
hold and
\begin{equation}
\label{e33}
\gamma(t)=\frac {c_1}{(1+t)^{q_1}}, \quad q_1\le 1; \quad c_1,q_1=const>0.
\end{equation}
Suppose that $\epsilon\in (0,c_1)$ is an arbitrary small fixed number,
$$\lambda\ge \left(\frac{c_0}{\epsilon}\right)^{1/(p-1)} \quad \text{ 
and}\quad \|u(0)\|\le \frac {1}{\lambda}.$$
Then the unique solution to \eqref{e1}-\eqref{e2} with $b(t)=0$  
exists on all of $\R_+$ and
\begin{equation}
\label{e34}
0\le \|u(t)\|\le \frac {1}{\lambda(1+t)^{c_1-\epsilon}},\qquad \forall 
t\ge 0. 
\end{equation}
}

Theorem 4 gives the size of the initial data, namely, $\|u(0)\|\le \frac 
{1}{\lambda}$, for which estimate \eqref{e34} holds. For a fixed 
nonlinearity $F(t,u)$, that is, for a fixed constant $c_0$ from  
assumption \eqref{e4}, the maximal size of $\|u(0)\|$ is determined by the 
minimal size of $\lambda$.

The minimal size of $\lambda$ is determined by the inequality 
$\lambda\ge \left(\frac{c_0}{\epsilon}\right)^{1/(p-1)}$, that is, by 
the maximal size of $\epsilon\in (0,c_1)$.  If $\epsilon<c_1$ and
$c_1-\epsilon$ is very small, then  
$\lambda>\lambda_{min}:= \left(\frac{c_0}{c_1}\right)^{1/(p-1)}$
and $\lambda$ can be chosen very close to $\lambda_{min}$. 

{\it Proof of Theorem 4.}
Let
\begin{equation}
\label{e35}
\mu(t)=\lambda (1+t)^\nu, \qquad \lambda, \nu=const>0.
\end{equation}
We will choose the constants $\lambda$ and $\nu$ later.
Inequality \eqref{e26} holds if
\begin{equation}
\label{e36}
\frac{c_0}{\lambda^{p-1} (1+t)^{(p-1)\nu}} +\frac {\nu}{1+t}\le \frac
{c_1}{(1+t)^{q_1}}, \qquad \forall t\ge 0.
\end{equation}
If
\begin{equation}
\label{e37}
q_1\le 1, \quad (p-1)\nu\ge q_1,
\end{equation}
then inequality \eqref{e36} holds if
\begin{equation}
\label{e38}
\frac {c_0}{\lambda^{p-1}} +\nu\le c_1.
\end{equation}
Let $\epsilon>0$ be an arbitrary small number. Choose
\begin{equation}
\label{e39}
\nu= c_1-\epsilon.
\end{equation}
Then inequality \eqref{e38}  holds if inequality \eqref{e30} holds. 
Inequality \eqref{e10}
holds because we have assumed in Theorem 4 that $\|u(0)\|\le \frac 1 
\lambda$.
Combining inequalities  \eqref{e30}, \eqref{e31} and \eqref{e11}, one 
obtains the desired estimate:
\begin{equation}
\label{e40}
0\le \|u(t)\|=g(t)\le \frac 1 {\lambda (1+t)^{c_1 -\epsilon}}, \qquad
\forall t\ge 0.
\end{equation}
Condition  \eqref{e30} holds for any fixed small $\epsilon>0$ if
$\lambda$ is sufficiently large. Condition  \eqref{e31} holds for any
fixed large $\lambda$ if $\|u_0\|$ is sufficiently small.

Theorem 4 is proved. \hfill $\Box$

Let us formulate a stability result in which we assume that 
$b(t)\not\equiv 0$.  
The function $b(t)$ has physical meaning of persistently acting 
perturbations.

{\bf Theorem 5.} {\it Let $b(t)\not\equiv 0$, conditions \eqref{e3}-
\eqref{e5} and \eqref{e33} hold, and 
\begin{equation}
\label{e41}
\beta(t)\le \frac {c_2}{(1+t)^{q_2}}, 
\end{equation}
where $c_2>0$ and $q_2>0$ are constants.
Assume that
\begin{equation}
\label{e42}
q_1\le \min\big(1, q_2-\nu, \nu(p-1)\big),
\end{equation}
and 
\begin{equation}
\label{e43}
c_2^{1-\frac 1 p}c_0^{\frac 1 p}(p-1)^{\frac 1 p}\frac p{p-1} +\nu\le c_1.
\end{equation}
Then problem \eqref{e1}-\eqref{e2} has a unique global solution $u(t)$,
and the following  estimate holds:
\begin{equation}
\label{e44}
\|u(t)\|\le \frac 1 {\lambda_0 (1+t)^\nu}, \qquad \forall t\ge 0,
\end{equation}
where $\lambda_0>0$ is a constant defined in \eqref{e49}.
}

{\it Proof of Theorem 5.} Let $g(t):=\|u(t)\|$. As in the proof
of Theorem 4, multiply \eqref{e1} by $u$, take the real part, use the 
assumptions of Theorem 5, and get the inequality:
\begin{equation}
\label{e45}
\dot{g}\le -\frac {c_1}{(1+t)^{q_1}}g +c_0g^p+\frac {c_2}{(1+t)^{q_2}}.
\end{equation}
Choose $\mu(t)$ by formula \eqref{e35}. Apply Theorem 1 to
inequality \eqref{e45}. Condition \eqref{e9} takes now the form
\begin{equation}
\label{e46}
\frac {c_0}{\lambda^{p-1} (1+t)^{(p-1)\nu}}+ \frac 
{\lambda c_2}{(1+t)^{q_2-\nu}}+ \frac \nu{1+t}\le \frac 
{c_1}{(1+t)^{q_1}}\quad \forall t\ge 0. 
\end{equation}
If assumption \eqref{e42} holds, then inequality \eqref{e46} holds
provided that it holds for $t=0$, that is, provided that
\begin{equation}
\label{e47}
\frac {c_0}{\lambda^{p-1}}+ \lambda c_2+ \nu \le c_1.
\end{equation} 
Condition  \eqref{e10} holds if
\begin{equation}
\label{e48}
g(0)\leq \frac 1 \lambda.
\end{equation} 
The function  $h(\lambda):=\frac {c_0}{\lambda^{p-1}}+ \lambda 
c_2$ attains its global minimum in the interval $[0,\infty)$
at the value 
\begin{equation}
\label{e49}
\lambda=\lambda_0:=\left(\frac {(p-1)c_0}{c_2}\right)^{1/p},
\ee
and this minimum
is equal to 
$$h_{min}=c_0^{\frac 1 p}c_2^{1-\frac 1 p}(p-1)^{\frac 1 
p}\frac p{p-1}.$$ 
Thus, substituting $\lambda=\lambda_0$ in formula \eqref{e47},  
one concludes that inequality \eqref{e47} holds if
the following inequality holds:
\begin{equation}
\label{e50}
c_0^{\frac 1 p}c_2^{1-\frac 1 p}(p-1)^{\frac 1 p}\frac p{p-1}+\nu\le c_1,
\end{equation}
while inequality \eqref{e48} holds if
\begin{equation}
\label{e51}
\|u(0)\|\le \frac 1 {\lambda_0}. 
\end{equation}
Therefore, by Theorem 1, if conditions \eqref{e50}-\eqref{e51} hold, then 
estimate \eqref{e11}
yields
\begin{equation}
\label{e52}
\|u(t)\|\le \frac 1 {\lambda_0 (1+t)^\nu}, \qquad \forall t\ge 0,
\end{equation}
where $\lambda_0$ is defined in \eqref{e49}.

Theorem 5 is proved. \hfill $\Box$

\end{document}